\documentclass[12pt,a4paper]{asl}

\usepackage{xspace}
\usepackage{dsfont}
\title{Indestructibility of Vop\v enka's Principle}

\author{Andrew D. Brooke-Taylor}
\revauthor{Brooke-Taylor, Andrew D.}

\address{Department of Mathematics\\
University of Bristol\\
University Walk\\
Bristol, BS8 1TW, UK}

\email{Andrew.Brooke-Taylor@bristol.ac.uk}

\thanks{This research was conducted at the University of Bristol 
with support from the 
Heilbronn Institute for Mathematical Research.}

\newcommand{\al}{\alpha}
\newcommand{\be}{\beta}
\newcommand{\ga}{\gamma}
\newcommand{\de}{\delta}
\newcommand{\ka}{\kappa}
\newcommand{\la}{\lambda}
\renewcommand{\phi}{\varphi}

\newcommand{\st}{\,|\,}
\newcommand{\forces}{\Vdash}
\newcommand{\sat}{\vDash}
\renewcommand{\P}{\mathbb{P}}
\newcommand{\Q}{\mathbb{Q}}
\newcommand{\restr}{\!\upharpoonright\!}
\DeclareMathOperator{\crit}{crit}
\DeclareMathOperator{\dom}{dom}
\newcommand{\HOD}{\textrm{HOD}}
\newcommand{\ORD}{\textrm{Ord}}
\DeclareMathOperator{\ot}{ot}
\DeclareMathOperator{\rank}{rank}
\DeclareMathOperator{\supp}{supp}
\DeclareMathOperator{\trcl}{trcl}
\newcommand{\calL}{\mathcal{L}}
\newcommand{\calM}{\mathcal{M}}
\newcommand{\calN}{\mathcal{N}}

\newcommand{\Vopenka}{Vop\v{e}nka\xspace}
\newcommand{\VPr}{Vop\v{e}nka's Principle\xspace}
\newcommand{\VP}{\textrm{VP}}

\newcommand{\Lstd}{\mathcal{L}_{\textrm{std}}}

\newcommand{\genval}[2]{#1_{#2}}

\newtheorem{thm}{Theorem}
\newtheorem{defn}[thm]{Definition}
\newtheorem{lemma}[thm]{Lemma}
\newtheorem{coroll}[thm]{Corollary}
\newtheorem{prop}[thm]{Proposition}

\newtheorem*{qn}{Open Question}

\newcommand{\Implies}{\Rightarrow}

\renewcommand{\iff}{\leftrightarrow}

\newcommand{\note}[1]{\relax}

\begin{document}

\begin{abstract}
We show that \VPr and \Vopenka cardinals are indestructible under 
reverse Easton forcing iterations of increasingly directed-closed
partial orders, without the need for any preparatory forcing.
As a consequence, we are able to prove the relative consistency of 
these large cardinal axioms with a variety of statements known to be 
independent of ZFC, such as the generalised continuum hypothesis, 
the existence of a definable well-order
of the universe, and the existence of morasses at many cardinals.
\end{abstract}

\maketitle

\section{Introduction}

\Vopenka's Principle is a large cardinal axiom that can readily be expressed in
either set- or category-theoretic terms.  It has received significant
attention in the latter context, yielding structural results for certain
important kinds of categories, as described in the final chapter of 
Ad\'amek and Rosick\'y's book~\cite{AdR:LPAC}.
This has led to the resolution
under the assumption of Vo\-p\v{e}nka's Principle of a long-standing open
question in algebraic topology:
Casacuberta, Scevenels and Smith~\cite{CSS:LCHL} 
have shown that if Vo\-p\v{e}nka's Principle holds, then
Bousfield localisation functors exist for all generalised cohomology theories.  

From a set-theoretic perspective, whilst it has received some attention
(for example in \cite{Bag:CnC},
\cite{Kan:VRP}, \cite{Pow:AHCVP}, 
\cite{Sato:DHLLC}, and \cite{SRK:SAIEE}), 
\VPr has generally been overshadowed by other large cardinal axioms,
and there are many natural questions regarding it that remain unanswered.
One aim of this article is to remedy the situation somewhat, 
providing the means for obtaining relative consistency results for
\VPr with various other statements known to be independent of ZFC.
Since \VPr lies beyond the scope of current inner model theory,
this entails an analysis of the interaction between \VPr and the other
standard technique for obtaining consistency results, namely,
Cohen's method of forcing.
%, which is used to get outer models of the statements in question.  

Specifically, we consider the common approach to obtaining consistency results
for very large cardinals by forcing, 
whereby one starts with a model of ZFC containing
the desired large cardinals, and forces other statements to hold, whilst 
preserving the large cardinal of interest.  In many cases this may be achieved
by Silver's technique of lifting embeddings, with a generic chosen
to contain a particular ``master condition'' that forces the large
cardinal to be preserved (see for example Cummings~\cite{Cum:IFE}).  
Laver~\cite{Lav:prep} showed that with a suitable
preparatory forcing, a supercompact cardinal $\kappa$ can be made 
indestructible under $\kappa$-directed closed forcing, in that \emph{any}
generic for such a forcing will preserve the supercompactness of $\kappa$.
Since then a variety of such indestructibility results have been obtained
--- see for example
\cite{Apt02:SCFLI}, \cite{ApH99:UI}, \cite{ApS10:EUI},
\cite{GiS:CISCQH}, \cite{Ham:LP}, \cite{HJo:ISU} and \cite{Joh08:SUCMI}.
%Gitik and Shelah~\cite{GiS:CISCQH} showed that strong cardinals may be
%analogously prepared to be indestructible under certain Prikry-type forcings.
%More recently, Hamkins~\cite{Ham:LP} has shown 
%that a wide range of large cardinals
%may be made indestructible under appropriate forcings
%by means of the so-called \emph{lottery preparation},
%and Johnstone~\cite{Joh08:SUCMI} used this technique to show that
%strongly unfoldable cardinals may also be made indestructible under 
%$\ka$-closed, $\ka$-proper
%forcings, a result extended to $\ka$-closed, $\ka^+$-preserving 
%forcings in work of Hamkins and Johnstone~\cite{HJo:ISU}.
In each case, a preparatory forcing is used to ensure that the large
cardinal is (or cardinals are) indestructible in the extension universe.

We shall show here that \VPr is in fact always
indestructible under a useful class of forcings,
with no preparation necessary.
Specifically, we show that reverse Easton iterations of increasingly 
directed-closed forcings preserve \VPr.
As a warm-up, we first show that \VPr is indestructible under small forcing,
analogously to the L\'evy--Solovay Theorem for measurable cardinals.
Key to both arguments is the fact that \VPr is witnessed by many embeddings,
so it is not important to lift any particular one.  
It suffices to lift \emph{some} embedding for each proper class of structures,
and this is forced by a dense set of conditions in each case.

We have just alluded to the fact that
the na\"{\i}ve statement of Vo\-p\v{e}n\-ka's Principle 
involves proper classes. 
A class theory such as von Neumann-Bernays-G\"odel or Morse-Kelley
might consequently
seem like the ``right'' context in which to study Vop\v{e}nka's Principle,
particularly as we shall have to be careful about use of 
the Global Axiom of Choice.
However, we accede to modern set-theoretic tastes, and work within
Zermelo-Fraenkel set theory.  There are two approaches to this.
One is to take ``classes'' to mean definable classes,
and consider \VPr to be an axiom schema,
as for example in the recent work of Bagaria~\cite{Bag:CnC}.
Alternatively,
one can consider inaccessible cardinals $\ka$ such that $V_\ka$ satisfies
a class theoretic form of \VPr 
when all members of $V_{\ka+1}$ are taken to be classes, as in 
Kanamori~\cite{Kan:THI}.  
We shall first prove our indestructibility result for such
\emph{\Vopenka cardinals}, 
in order that the 
main ideas might not be obscured by the technicalities involved with 
dealing with definable proper classes.  
That done, we shall in Section~\ref{DVP} address 
those technicalities and how they may be overcome to yield the corresponding 
results for the definable class version of \VPr.

\section{Definitions}

In order to define \Vopenka's Principle, we
fix some model-theoretic notation.
It will be convenient to refer to 
a standard language
$\Lstd$ 
with one binary relation symbol $\epsilon$ and one
unary relation symbol $R$.  
Since any number of relations of any arity can be encoded in a single
binary relation 
(see for example Pultr and Trnkov\'a~\cite[Theorem~II.5.3]{PuT:CATRGSC}), 
this is tantamount to considering
all languages.
%We use caligraphic Latin letters for set-sized models, and
Unless otherwise specified, $M$ will denote the domain of any model denoted
$\calM$, and $N$ will denote the domain of any model denoted $\calN$.

\begin{defn}\label{VPA}
Let $\ka$ be an inaccessible cardinal and let 
$A$ be a subset of $V_\ka$ of cardinality $\ka$ such that every 
element of $A$ is an $\Lstd$-structure.  
We denote by $\VP(A)$ the statement that 
there are distinct $\calM$ and $\calN$ in $A$ such that there exists 
an $\Lstd$-elementary embedding $j:\calM\to\calN$.
\end{defn}

\begin{defn}\label{VCard}
A cardinal $\ka$ is a \emph{\Vopenka cardinal} if $\ka$ is inaccessible,
and for every $A\subset V_\ka$ of cardinality $\ka$ such that every element of
$A$ is an $\Lstd$-structure, $\VP(A)$ holds.
\end{defn}
\begin{defn}\label{VP}
\VPr is the axiom schema that states the following.
\begin{description}
\item[VP] For every (definable) proper class of $\Lstd$-structures, 
there are distinct members
$\calM$ and $\calN$ of the class
such that there is an $\Lstd$-elementary embedding 
$j:\calM\to\calN.$
\end{description}
%\begin{multline*}
%\Big(\forall x(\phi(x)\implies x\text{ is an }\Lstd\text{-structure})\land\\
%\forall \al(\al\in\ORD\implies \exists x(\rank(x)>\al\land\phi(x))\Big)
%\end{multline*}
%\vspace{-7mm}
%\begin{multline*}
%\implies\exists x\exists y\exists j\Big(\phi(x)\land\phi(y)\land\\
%j\text{ is an non-identity }
%\Lstd\text{-elementary embedding from $x$ to $y$}\Big).
%\end{multline*}
\end{defn}
An alternative way to define \Vopenka cardinals and
\VPr is to remove the above distinctness requirements on
$\calM$ and $\calN$, and only require that the embedding $j$ be non-trivial;
we have chosen to follow Solovay, Reinhardt and Kanamori~\cite{SRK:SAIEE}
in this regard.
Since rigid graphs can be constructed of any cardinality 
using the axiom of choice (see~\cite{VPH:RRAS}),
\note{Qn: find a model of ZF and a set on which there is no rigid graph.}
and structures with two binary relations can be encoded into graphs in
a way that respects homomorphisms
(see for example~\cite{PuT:CATRGSC}), 
these formulations are equivalent under the
assumption of global choice.
In the definable class setting, global choice is equivalent to $V=\HOD$.
We shall avoid using this assumption, and show in Section~\ref{DVP} 
that in fact $V=\HOD$ may be forced while preserving \VPr
(but note that the form of \VPr with which we work is the
\emph{stronger} version of the two alternatives in the absence of $V=\HOD$).

We now focus on \Vopenka cardinals, 
leaving the details of the corresponding results for
\VPr to Section~\ref{DVP}.
It will be convenient to have at our disposal another equivalent but
more restricted characterisation
of \Vopenka cardinals

\begin{defn}
For any language $\calL$,
an \emph{ordinal $\calL$-structure} is an $\calL$-structure 
with domain an ordinal.
\end{defn}

\begin{lemma}\label{VPoLs}
%\renewcommand\theenumi {\roman{enumi}}
%\begin{enumerate}
%\item\label{VPoLska} 
For any inaccessible cardinal $\ka$, 
$\ka$ is \Vopenka if and
only if for every set $B\subset V_\ka$ of cardinality $\ka$
of ordinal $\Lstd$-structures,
there exist distinct $\calM$ and $\calN$ in $B$ such that there is an
elementary embedding $j:\calM\to \calN$.
%\item\label{VPoLsCl} If $V=\HOD$, then \VPr holds if and only if 
%for every proper class $B$ of ordinal $\Lstd$-structures,
%there exist distinct $\calM$ and $\calN$ in $B$ such that there is an
%elementary embedding $j:\calM\to \calN$.
%\end{enumerate}
%\renewcommand\theenumi {\arabic{enumi}}
\end{lemma}
\begin{proof}
Any $A$ as in Definition~\ref{VCard} may clearly be converted to a corresponding
set $B$ of ordinal $\Lstd$-structures
using a choice function to choose for each element $\calM$ of $A$
a unique representatives from the set of 
ordinal $\Lstd$-structures isomorphic to $\calM$.
An elementary embedding between distinct members of $B$ then 
corresponds to an elementary embedding between distinct members of $A$.
\end{proof}

Note that a na\"{\i}ve recasting of this proof in terms of definable
classes would require $V=\HOD$.  
However, we shall see in Section~\ref{DVP} that this
is not necessary for the corresponding lemma to hold.

%For the proof of our main theorem, Theorem~\ref{REInd}, 
We will also need the following characterisation of 
\Vopenka cardinals, which is a slight refinement of one
from Kanamori~\cite{Kan:THI}.

\begin{defn}\label{extblBelkaA}
Let $A$ be a set and $\eta$ an ordinal less than $\ka$.  
A cardinal $\al<\eta$ is
\emph{$\eta$-extendible below $\ka$ for $A$} 
if there is some $\zeta<\ka$ and an elementary
embedding
$j:\langle V_\eta,\in,A\cap V_\eta\rangle\to
\langle V_\zeta,\in,A\cap V_\zeta\rangle$
with critical point $\al$ and $j(\al)>\eta$.
A cardinal $\al<\ka$ is \emph{extendible below $\ka$ for $A$}
if it is $\eta$-extendible below $\ka$ for $A$ for all $\eta$
strictly between $\al$ and $\ka$.
\end{defn}

\begin{prop}\label{VopCardExtA}
The following are equivalent for inaccessible cardinals $\ka$:
\begin{enumerate}
\item $\ka$ is a \Vopenka cardinal,
\item for every $A\subset V_\ka$, there is a cardinal $\al<\ka$ that 
is extendible below $\ka$ for $A$.
\end{enumerate}
\end{prop}
\begin{proof}
The proof of 
%Theorem~6.9 of Solovay, Reinhardt and Ka\-na\-mori~\cite{SRK:SAIEE} 
Exercise~24.19 of Kanamori~\cite{Kan:THI}
also proves this ``below $\ka$'' refinement.
\end{proof}

We mostly follow the notational conventions of Kunen~\cite{Kun:ST}
for forcing concepts; in particular, $q\leq p$ shall mean that 
$q$ is a stronger condition than $p$, 
and
for any set $x$ in the ground model, we shall denote by $\check x$ the 
canonical name for $x$ in the extension, 
$\{\langle\check y,\mathds{1}
\rangle\st y\in x\}$.

\section{$\ka^+$-distributive forcing and $\square$}

The following Proposition is immediate from the definition of 
\Vopenka cardinals.

\begin{prop}\label{distrib}
If $\ka$ is a \Vopenka cardinal 
and $\P$ is a $\ka^+$-distrib\-utive partial order
(that is, forcing with $\P$ adds no new subsets of $\ka$),
then $\ka$ remains a \Vopenka cardinal after forcing with $\P$.
\hfill\qedsymbol
\end{prop}

Whilst Proposition~\ref{distrib} is entirely trivial, it can be used to 
make the following interesting observation.
Recall that for any uncountable cardinal $\ka$, a \emph{$\square_\ka$-sequence}
is a sequence $\langle C_\al\st\al\in\textrm{lim}\cap\ka^+\rangle$ such that
for all $\al\in\textrm{lim}\cap\ka^+$,
\begin{itemize}
\item $C_\al$ is a club in $\al$,
\item $\ot(C_\al)\leq\ka$,
\item if $\be\in\lim(C_\al)$ then $C_\be=C_\al\cap\be$.
\end{itemize}
The statement that there exists a $\square_\ka$-sequence is denoted 
simply by
$\square_\ka$.

\begin{coroll}
It is relatively consistent with ``$\ka$ is a \Vopenka cardinal'' 
that $\square_\ka$ holds.
\end{coroll}
\begin{proof}
The usual forcing (due to Jensen) to make $\square_\ka$ hold, 
in which the conditions are
partial $\square_\ka$ sequences, is $<\ka^+$-strategically closed
(see for example Cummings~\cite[Sections~5 and 6]{Cum:IFE}),
and in particular is $\ka^+$-distributive.
\end{proof}

This contrasts with the result of Jensen (see Friedman~\cite{SDF:LCL})
that if $\ka$ is subcompact then
$\square_\ka$ fails: subcompact cardinals are consistency-wise
weaker than \Vopenka cardinals.
For further discussion of $\square_\ka$ and its failure 
for subcompact and related cardinals $\ka$, see
%the final section of 
Cummings and Schimmerling~\cite[Section~6]{CuS:ISq}.

Proposition~\ref{distrib} will also be important for showing that
\Vopenka cardinals $\kappa$ are preserved by appropriate forcing iterations that
go beyond $\kappa$.

\begin{coroll}\label{2StepIt}
If $\ka$ is a cardinals and $\P*\dot\Q$ is a forcing iteration such that
\[
\forces_\P\check\kappa\textrm{ is a \Vopenka cardinal }\land
\dot\Q\textrm{ is $\check\kappa^+$-distributive }
\]
then $\forces_{\P*\dot\Q}\check\kappa\textrm{ is a \Vopenka cardinal }$.
\hfill\qedsymbol
\end{coroll}

\section{Small forcing}

It has become part of the set-theoretic folklore that small forcing preserves
most large cardinals, where by ``small'' we mean 
of cardinality less than the large cardinal in question.
This stems from the original result of L\'evy and Solovay \cite{LvS:MCCH}
that small forcing preserves measurable cardinals, and the fact that
most strong large cardinal properties can be expressed similarly to measurable
cardinals with a witnessing elementary embedding.
Whilst the definition of \Vopenka cardinals does involve elementary
embeddings, 
%so that one might think that they should be preserved by small forcings, 
it is not immediately clear that the usual argument will extend to this case.
In this section we shall confirm that these large cardinal properties
are preserved by small forcing,
and in the process set the scene for later sections.
Of course, this could be considered to be a special case of our main theorem,
but there are tricks we can use to simplify the argument significantly 
in this context for the reader not interested in the full reverse Easton
iteration result.

We need the following well-known, basic large cardinal preservation result.
\begin{lemma}\label{InaccPres}
If $\ka$ is an inaccessible cardinal in $V$, $\P$ is a partial order of
cardinality less than $\ka$, and $G$ is $\P$-generic over $V$,
then $\ka$ is inaccessible in $V[G]$.
\end{lemma}
\begin{proof}
For any partial order $\P$,
$\P$ is $|\P|^+$-cc, and hence preserves
cofinalities greater than $|\P|$ and the continuum function 
$\la\mapsto 2^\la$ for $\la\geq|\P|$.
\end{proof}

A key point in the proof of Theorem~\ref{smallpres}, 
and indeed our later theorems,
is that the embeddings in which we are interested for 
\Vopenka cardinals need not
respect the given subset $A$ of $V_\ka$, but rather be elementary between
two elements of $A$;
and yet no particular element is especially important, as 
there will be many embeddings witnessing \VPr for each subset $A$ of $V_\ka$.
This means that we can replace a name $\dot A$ for a subset of $V_\ka$
by one consisting of particularly nice names for its elements, giving
us much more control.

%In proving the main result of this section, %Theorem~\ref{smallpres} 
%as well as later, % theorems, 
%it will 
%be useful to be able to consider particularly nice names for structures.

\begin{defn}
Let $\P$ be a forcing partial order and $\calL$ a relational language.
A \emph{nice $\P$-name for an ordinal $\calL$-structure} is a $\P$-name
$\sigma$ of the following form:
\begin{enumerate}
\item $\sigma$ is the canonical name for %a $1+|\calL|$-tuple 
an ordinal $\calL$-structure
\mbox{$\langle\gamma_\sigma,R^\sigma\st$}\mbox{$R\in\calL\rangle$} in $V[G]$
with components named by $\check\ga_\sigma$ and $\dot R^\sigma$ for
$R\in\calL$,
%$\langle\check\gamma_\sigma,\dot E,\dot R\rangle$,
%\item $\gamma_\sigma$ is an ordinal,
\item if $R\in\calL$ is $n$-ary, 
the name $\dot R^\sigma$ for $R^\sigma$ is of the form
\[
\bigcup_{(\be_1,\ldots,\be_n)\in \ga_\sigma^n}
\{\check{(\be_1,\ldots,\be_n)}\}\times A^R_{(\be_1,\ldots,\be_n)}
\]
where
$A^R_{(\be_1,\ldots,\be_n)}$ is an antichain in $\P$
for each $(\be_1,\ldots,\be_n)\in\ga_\sigma^n$.
%\item $\dot E$ is a name of the form
%\[
%\bigcup_{(\be_1,\be_2)\in \ga_\sigma\times\ga_\sigma}
%\{\check{(\be_1,\be_2)}\}\times A^E_{(\be_1,\be_2)}
%\]
%where
%$A^E_{(\be_1,\be_2)}$ is an antichain
%for each $(\be_1,\be_2)\in\ga_\sigma\times\ga_\sigma$,
%\item $\dot R$ is a name of the form
%\[
%\bigcup_{\be<\ga_\sigma}\{\check\beta\}\times A^R_\be
%\]
%where $A^R_\be$ is an antichain
%for each $\be\in\ga_\sigma$.
\end{enumerate}
\end{defn}

\begin{lemma}\label{nicenames}
Let $\calL$ be a language, $\P$ a partially ordered set, 
$\dot A$ a $\P$-name, and $p$ an element of
$\P$ such that 
\[
p\forces\forall a\in\dot A(a\textnormal{ is an ordinal $\calL$-structure}).  
\]
Then there is a $\P$-name $\dot B$ such that $p\forces\dot A=\dot B$
and for every element
$\langle\sigma,q\rangle$ of $\dot B$, 
$q\leq p$ and $\sigma$ is a nice $\P$-name for 
an ordinal $\Lstd$-structure.
\end{lemma}
\begin{proof}
The argument is a fairly typical ``nice-name'' construction.  
For each $\langle\dot a,r\rangle\in\dot A$,
let 
\[
Q_{\dot a,r}=\{q\in\P\st q\leq p\land q\leq r\land
\exists\ga(q\forces\dom(\dot a)=\ga)\}.
\]
%Of course $Q_{\dot a,r}$ may be a proper class if $\P$ is.
Suppose we have $q\in Q_{\dot a,r}$ and $\ga$ such that 
$q\forces\dom(\dot a)=\ga$. 
For each $n$-ary relation symbol $R\in\calL$
and each $(\be_1,\ldots,\be_n)\in\ga^n$, 
we can choose
an antichain $A_{(\be_1,\ldots,\be_n)}$
below $q$ such that for each $s\in A_{(\be_1,\ldots,\be_n)}$ we have
\mbox{$s\forces\check{(\be_1,\ldots,\be_n)}\in R^{\dot a}$,}
and such that $A_{(\be_1,\ldots,\be_n)}$
is maximal with this property
(here $R^{\dot a}$ denotes the canonical name for the interpretation of the 
relation $R$ in the structure named by $\dot a$). 
Then setting
$\dot R^{q,\dot a}_{(\be_1,\ldots,\be_n)}=
\{(\be_1,\ldots,\be_n)\}\times A_{(\be_1,\ldots,\be_n)}$,
we have by standard arguments that
\[
q\forces\check{(\be_1,\ldots,\be_n)}\in R^{\dot a}\iff
\check{(\be_1,\ldots,\be_n)}\in\dot R^{q,\dot a}_{(\be_1,\ldots,\be_n)}.
\]
Taking 
$\dot R^{\sigma_{\dot a,r,q}}=
\bigcup_{(\be_1,\ldots,\be_n)\in\ga^n}\dot R^{q,\dot a}_{(\be_1,\ldots,\be_n)}$,
and $\ga_{\sigma_{\dot a,r,q}}=\gamma$,
we obtain a nice $\P$-name 
$\sigma_{\dot a,r,q}$
for an ordinal $\calL$-structure 
such that $q\forces\sigma_{\dot a,r,q}=\dot a$.
Taking $\dot B=\bigcup_{\langle\dot a,r\rangle\in\dot A}
\{\langle\sigma_{\dot a,r,q},q\rangle\st q\in Q_{\dot a,r}\}$,
we have $p\forces\dot A=\dot B$, as required.
\end{proof}

\begin{thm}\label{smallpres}
Suppose $\ka$ is a \Vopenka cardinal in $V$, 
$\P$ is a partially ordered set of
cardinality less than $\ka$, and $G$ is $\P$-generic over $V$.  
Then $\ka$ is \Vopenka in $V[G]$.
\end{thm}
\begin{proof}
By replacing $\P$ with an isomorphic partial order if necessary, we may 
assume for convenience that the underlying set of $\P$ is the cardinal
$|\P|<\ka$.
We know 
from Lemma~\ref{InaccPres} that $\ka$ is inaccessible in $V[G]$, 
so it suffices to show that
for any set 
$A$ of cardinality $\ka$ of ordinal $\Lstd$-structures in 
$(V_\ka)^{V[G]}$, 
there are distinct elements $\calM$ and $\calN$ of $A$ with 
an elementary embedding $j:\calM\to\calN$ in $V[G]$.
Let $\dot A$ be a $\P$-name for $A$, and let $p\in G$ be such that
\[
p\forces(\dot A\subset V_{\check\kappa})\land(|\dot A|=\check\kappa)\land
\forall a\in\dot A(a\textrm{ is an ordinal $\Lstd$-structure}).
\]
We shall show that $p\forces\VP(\dot A)$.

Our approach will be
to show that it is dense below $p$ to force 
$\VP(\dot A)$,
for then $p$ will also force $\VP(\dot A)$.  
So suppose we have $r\leq p$; %, and consider $\dot B^{\textrm{nice}}_r$.
by Lemma~\ref{nicenames} there is a name $\dot B_r$
such that $r\forces\dot A=\dot B_r$, and for every element
$\langle\sigma,q\rangle$
of $\dot B_r$, $q\leq r$ and $\sigma$ is a nice $\P$-name for
an ordinal $\Lstd$-structure.
To avoid concerns about the distinctness of the $\calM$ and $\calN$ we find,
we may thin out $\dot B_r$ to a name $\dot C$, 
still with $|\dot C|=\ka$, such
that if $\langle\sigma_0,q_0\rangle\neq\langle\sigma_1,q_1\rangle$ are
both in $\dot C$, then $\ga_{\sigma_0}\neq\ga_{\sigma_1}$,
and so certainly $\sigma_0^{G'}\neq\sigma_1^{G'}$ in any generic extension
by a $\P$-generic ${G'}$ containing $r$.

Let $R_\P$ be a rigid binary relation on $|\P|$, that is, one admitting 
no non-identity endomorphism; see \cite{AdR:LPAC},
\cite{PuT:CATRGSC}, or the original paper of Vop\v{e}nka, Pultr and
Hedrl\'{\i}n~\cite{VPH:RRAS} for such a construction.
Now in $V$ consider the set
\[
D=\left\{\left\langle
H_{\max(|\trcl(\sigma)|,|\P|)^+},
\in,\langle\sigma,q,R_\P\rangle
\right\rangle
\st\langle\sigma,q\rangle\in\dot C\right\}.
\]
Since $\ka$ is a \Vopenka cardinal in $V$, there are 
$\calM\neq\calN$ in $D$
with an elementary embedding $j:\calM\to\calN$.
Suppose
\[
\calM=\langle H_\al,\in,\langle\sigma_\calM,q_\calM,R_\P\rangle\rangle
\]
and
\[
\calN=\langle H_\be,\in,\langle\sigma_\calN,q_\calN,R_\P\rangle\rangle;
\]
we shall show that 
in any generic extension $V[{G'}]$ by a $\P$-generic ${G'}$ containing $r$,
$j\restr\gamma_{\sigma_\calM}:\ga_{\sigma_\calM}\to\ga_{\sigma_\calN}$
is elementary 
when considered as a map $\sigma_\calM^{G'}\to\sigma_\calN^{G'}$. 

%Let $\dot s$ be the canonical name for the satisfaction function for 
%$\sigma_\calM^{G'}$, that is,
%$s:\omega\times\ga_{\sigma_\calM}^{<\omega}\to 2$ such that
%$s(i,(x_0,\ldots,x_n))=1$ if and only if 
%$\sigma_\calM^{G'}\sat\varphi_i(x_0,\ldots,x_n)$,
%where $\varphi_i$ is the formula with G\"odel number $i$.
%Indeed, 
For any $\Lstd$-formula $\varphi$, 
$\sigma_\calM^{G'}\sat\varphi(\al_1,\ldots,\al_n)$ 
\note{do we really want to use $\Lstd$?}
if and only if there is some $q\in {G'}$
such that 
\begin{equation}
q\forces(\sigma_\calM\sat\varphi(\check\al_1,\ldots,\check\al_n)).\tag{$*$}
\end{equation}
Since satisfaction (for set models) is $\Sigma_1$-definable,
the statement ($*$) is also $\Sigma_1$ for models containing $\P$; 
indeed, it can be written in the form
\begin{align*}
\exists\dot s&\exists F\exists X\big(\\
&F \textrm{ is the characteristic function of the $\forces$ relation} \\
&\quad\textrm{for $\Delta_0$ formulae on the (sufficiently large) 
transitive set }X, \\
&\textrm{and }F\textrm{ witnesses that }\\
&\quad q\forces\textrm{``$\dot s$ is the characteristic function of the relation $\sigma_\calM\sat\,$''}, \\
&\textrm{and }F\textrm{ witnesses that}\\
&\quad q\forces\textrm{``$\dot s$ witnesses that $\sigma_\calM\sat\varphi(\check\al_1,\ldots,\check\al_n)$''}\ \big)
\end{align*}
%\begin{align*}
%\exists\dot s\exists F\exists X(&F
%\textrm{ is the characteristic function of the $\forces$ relation} \\
%&\quad\textrm{for $\Delta_0$ formulae on the (sufficiently large) 
%transitive set }X, \\
%&\textrm{and }F\textrm{ witnesses that }\\
%&\quad q\forces\textrm{``$\dot s$ is the characteristic function of the relation $\sigma_\calM\sat\,$''}, \\
%&\textrm{and }F\textrm{ witnesses that}\\
%&\quad q\forces\textrm{``$\dot s$ witnesses that $\sigma_\calM\sat\varphi(\check\al_1,\ldots,\check\al_n)$''})
%\end{align*}
where the main parenthesised part is $\Delta_0$.
As originally shown by L\'evy~\cite{Lvy:HFST}, 
one can prove with a L\"owenheim-Skolem argument
that $H_\la$ is $\Sigma_1$-elementary in $V$ for any uncountable cardinal $\la$.
In particular, $H_\al$, and hence $\calM$, is correct for the statement ($*$).
%\mbox{$q\forces(\sigma_\calM\sat\varphi(\check\al_1,\ldots,\check\al_n))$.}
By elementarity and the $\Sigma_1$-correctness of $\calN$, we thus have 
\[
q\forces(\sigma_\calM\sat\varphi(\check\al_1,\ldots,\check\al_n))\iff
j(q)\forces
(\sigma_\calN\sat\varphi({j(\check\al_1)},\ldots,{j(\check\al_n)})).
\]
By the rigidity of $R_\P$, $j(q)=q$ for any $q\in\P$, and we may conclude 
that for any $\P$-generic $G'$ containing $r$,
$\sigma_\calM^{G'}\sat\varphi(\al_1,\ldots,\al_n)$ if and only if
$\sigma_\calN^{G'}\sat\varphi(j(\al_1),\ldots,j(\al_n))$.

Finally, 
$q_\calM=q_\calN$ by rigidity once more, and if $q_\calM\in G'$, then
$\sigma_\calM^{G'}$ and $\sigma_\calN^{G'}$ are elements of 
$\dot C^{G'}\subset(\dot B_r)^{G'}=\dot A^{G'}$.
Hence,
$q_\calM\forces\VP(\dot A)$.
%\begin{multline*}
%q_\calM\forces\exists\calM\in{\dot B}\,\exists\calN\in{\dot B}\\
%(\calM\neq\calN
%\land\exists j(\textrm{$j$ is elementary from $\calM$ to $\calN$}))
%\end{multline*}
%\begin{multline*}
%q_\calM\forces\textrm{``There exist $\calM\neq\calN$ in $\dot B$
%such that} \\ \textrm{there is an elementary embedding $j:\calM\to\calN$.''}
%\end{multline*}
But now $q_\calM\leq r$ by the construction of $\dot B_r$.
Therefore, it is dense below $p$ to force $\VP(\dot A)$, and so 
$p\forces\VP(\dot A)$.
\end{proof}

\section{Reverse Easton iterations}

Having shown above that 
Vo\-p\v{e}n\-ka cardinals are preserved by small forcing, 
we modify the argument to show that they are moreover preserved 
by all generics for typical $\ka$-length iterations.
This contrasts with the situation for most strong large cardinals,
which are preserved only when generics are carefully chosen to contain
suitable master conditions.
The key idea remains the same as in the previous section: 
because we do not have to preserve all of
the embeddings present in the ground model, a density argument, which
can be carried out without extra assumptions or preparation, will suffice.

Recall that a forcing iteration has \emph{Easton support} if 
direct limits are taken at inaccessible cardinal stages and inverse limits
at other limit stages.
We call a forcing iteration (possibly of class length) a
\emph{reverse Easton iteration} if it has Easton support.
See Cummings~\cite{Cum:IFE} for more on such iterations.

The precise statement that we shall prove for \Vopenka cardinals
is the following.
\begin{thm}\label{REInd}
Let $\ka$ be a \Vopenka cardinal.
Suppose $\langle P_\al\st\al\leq\ka\rangle$ 
is the 
reverse Easton iteration of 
$\langle\dot Q_\al\st\al<\ka\rangle$,
where each $\dot Q_\al$ has cardinality less than $\ka$, and
for every $\ga<\ka$, there is an $\eta_0$ such that for all
$\eta\geq\eta_0$,
\[
\forces_{P_\eta}\dot Q_\eta\text{ is $\ga$-directed-closed.}
\]
%and for an unbounded set of inaccessible cardinals $\iota$ in $\ka$, 
%$|P_\iota|\leq\iota$.
Then 
\[
\forces_{P_\ka}\ka\text{ is a \Vopenka cardinal.}
\]
\end{thm}
Note also that a full class-length iteration will often preserve all
\Vopenka cardinals by Theorem~\ref{REInd} and 
Corollary~\ref{2StepIt}, so long as the tail of the iteration from any
\Vopenka cardinal $\ka$ is $\ka^+$-directed closed.

Since a direct limit is taken at $\ka$, we can and will
identify conditions in $P_\ka$ with conditions in 
$\bigcup_{\al<\ka}P_\al$.
Further, we observe that the ``smallness'' requirement on the names
$\dot Q_\al$ results in a certain amount of 
closure with respect to features of the forcing iteration 
being reflected downwards.
\begin{defn}\label{PkaRefl}
Let 
$\langle P_\al\st\al\leq\ka\rangle$ 
be a forcing iteration as in the statement of Theorem~\ref{REInd}.
We say that a Mahlo cardinal $\de<\ka$ is
\emph{$P_\ka$-reflecting} if
\begin{enumerate}
\item
$|P_\de|$ is at most $\de$, and 
\item
for all $\eta\geq\de$, 
$\forces_{P_\eta}\dot Q_\eta\textnormal{ is $\de$-directed-closed.}$
\end{enumerate}
\end{defn}
\begin{lemma}\label{PClosRefl}
Suppose $\ka$ is a \Vopenka cardinal, and 
$\langle P_\al\st\al\leq\ka\rangle$ 
is a forcing iteration as in the statement of Theorem~\ref{REInd}.
Then the set of $P_\ka$-reflecting Mahlo cardinals 
is stationary in $\ka$.
\end{lemma}
\begin{proof}
\Vopenka cardinals are 2-Mahlo --- 
see Kanamori~\cite[Corollary~24.17]{Kan:THI}.
It therefore suffices to show that the set of cardinals 
\[
\{\ga<\ka\st
\Big|\bigcup_{\al<\ga}P_\al\Big|\leq\ga\land
\forall\eta\geq\ga(\forces_{P_\eta}\dot Q_\eta\text{ is $\iota$-directed-closed.})\}
\] 
is closed unbounded in $\ka$.  
But this follows from a standard closure argument.
\end{proof}
For every $\ga<\ka$,
we shall denote by $\de(\ga)$ the least $P_\ka$-reflecting Mahlo cardinal
strictly greater than $\ga$.

As in the previous section, an important part of the proof is that we can be
very selective about the kind of names with which we work.

\begin{defn}
Let
$P_\ka$ be a forcing iteration as in the statement of
Theorem~\ref{REInd}.  
A \emph{nice local $P_\ka$-name for an ordinal $\calL$-structure} 
is a name $\sigma$ satifying the following requirements:
\begin{enumerate}
\item $\sigma$ is the canonical name for an ordinal $\calL$-structure
$\langle\gamma_\sigma,R^\sigma\st$\mbox{$R\in\calL\rangle$} in $V[G]$
with components named by $\check\ga_\sigma$ and $\dot R^\sigma$ for
$R\in\calL$,
\item \label{locality}
%if $\de$ is the least Mahlo cardinal strictly greater than $\ga_\sigma$ 
%such that $|P_\de|\leq\de$ and
%\[
%\eta\geq\de\implies
%\ \forces_{P_\eta}\dot Q_\eta\text{ is $\de$-directed-closed}
%\]
%then 
for every $n$-ary $R\in\calL$, 
the name $\dot R^\sigma$ is a $P_{\de(\ga_\sigma)}$-name for a subset of $\ga_\sigma^n$.
\end{enumerate}
\end{defn}

\begin{lemma}\label{nicelocnames}
For any finite language $\calL$ and forcing iteration $P_\ka$ as in
the statement of Theorem~\ref{REInd}, 
given a name $\dot A$ and a condition $p\in P_\ka$ such that 
\[
p\forces\dot A\text{ is a set of ordinal $\calL$-structures,}
\]
there is a name $\dot B$ such that $p\forces\dot A=\dot B$, 
and for every element $\langle\sigma,q\rangle$ of $\dot B$, 
$\sigma$ is a nice local name for an ordinal $\calL$-structure 
and $q\leq p$.
\end{lemma}
Finiteness of $\calL$ is a stronger requirement than is necessary,
but it is convenient and suffices for our purposes.
\begin{proof}
The proof is much like that for Lemma~\ref{nicenames}, 
except that we 
%are even more picky about which conditions $q$ we use.
dictate that
the conditions $q$ that we use should satisfy more
stringent requirements.
For each $\langle\dot a,r\rangle\in\dot A$, let
\begin{multline*}
Q_{\dot a,r} = \Big\{q\in P_\ka\st
q\leq p\land q\leq r\land
\exists\ga\Big(
(q\forces\dom(\dot a)=\ga)\land\\ % \supp(q)\geq\de(\ga)\land\\
%&\qquad\exists\de
%\big(
%\forall\zeta(\zeta\geq\de\implies\ 
%\forces_{P_\zeta}\dot Q_\zeta\text{ is $\ga^+$-directed-closed})\\
\forall R\in\calL\exists\dot R\big(
\dot R\text{ is a $P_{\de(\ga)}$-name}
\land(q\forces R^{\dot a}=\dot R)\big)
\Big)\Big\}
%&\land\dot E\text{ is a $P_\de$-name}\\
%&\land(q\force
\end{multline*}
%\begin{align*}
%Q_{\dot a,r} = \Big\{q\in P_\ka\st&
%q\leq p\land q\leq r\land\\
%&\exists\ga\Big(
%(q\forces\dom(\dot a)=\ga)\land\\ % \supp(q)\geq\de(\ga)\land\\
%%&\qquad\exists\de
%%\big(
%%\forall\zeta(\zeta\geq\de\implies\ 
%%\forces_{P_\zeta}\dot Q_\zeta\text{ is $\ga^+$-directed-closed})\\
%&\qquad\forall R\in\calL\exists\dot R\big(
%\dot R\text{ is a $P_{\de(\ga)}$-name}\\
%&\qquad\qquad\qquad\quad
%\land(q\forces R^{\dot a}=\dot R)\big)
%\Big)\Big\}
%%&\land\dot E\text{ is a $P_\de$-name}\\
%%&\land(q\force
%\end{align*}
%\note{supp(q) can't cover all of $\de(\ga)$}
where %$\supp(q)$ denotes the least ordinal $\al$ such that $q\in P_\al$,
$R^{\dot a}$ denotes the 
canonical name for the interpretation of $R$ in 
$\genval{\dot a}{G}$, and as is standard practice,
we abuse notation regarding for which partial order any
given name is actually a name.
%s are names with by neglecting the function taking
%$P_{\de(\ga)}$-names to $P_{\sup(\supp(q))}$-names induced by the inclusion
%$P_{\de(\ga)}\hookrightarrow P_{\sup(\supp(q))}$.
This set $Q_{\dot a, r}$ will be dense for conditions below both $p$ and $r$,
as $P_\ka$ may be factorised as \mbox{$P_{\de(\ga)}*\dot P^{[\de(\ga),\ka)}$}
with 
\[
\forces_{P_{\de(\ga)}}P^{[\de(\ga),\ka)}\text{ is $\de(\ga)$-directed-closed}
\]
(by for example Corollary~2.4 and Theorem~5.5 of Baumgartner~\cite{Bau:IF})
and so in particular every subset of $\ga$ in the extension can be named by a 
$P_{\de(\ga)}$-name.
Constructing names $\sigma_{\dot a,r,q}$ from the $\ga$ and names
$\dot R$ from the definition of $Q_{\dot a, r}$, the name 
$\dot B=\{\langle\sigma_{\dot a,r,q},q\rangle\st
\langle\dot a,r\rangle\in\dot A\}$ will be as desired.
\end{proof}

\begin{proof}[Proof of Theorem~\ref{REInd}]
Let $\ka$ and $P_\ka$ be as in the statement of the theorem, and 
let $G$ be $P_\ka$-generic over $V$;
we shall show that $\ka$ remains \Vopenka in $V[G]$.
Suppose that %$G$ is $P_\ka$-generic, 
$\dot A$ is a $P_\ka$-name,
and $p\in G$ is such that
\[
p\forces\dot A\subset V_{\check\ka}\land|\dot A|=\check\ka\land
\dot A\text{ is a set of ordinal $\calL$-structures in }V_{\check\ka}.
\]
Let $\dot B$ be as in Lemma~\ref{nicelocnames}.
Using the Proposition~\ref{VopCardExtA} characterisation of 
\Vopenka cardinals, let $\al<\ka$ be extendible below $\ka$ for $\dot B$
in $V$.
Let $\xi$ be 
the least $P_\ka$-reflecting Mahlo cardinal such that 
there is a $\langle\sigma,q\rangle\in\dot B\smallsetminus V_\al$
where $q\in G\cap P_\xi$ and
$\sigma$ names an ordinal $\Lstd$-structure
$\langle\gamma,E,R\rangle$ with
$\al\leq\ga<\xi$
(whence $\dot E$ and $\dot R$ 
%(as in $\sigma$) 
will be $P_\xi$-names).
%and $\de<\xi$,
%$|P_\xi|\leq\xi$, and for all $\eta\geq\xi$, 
%\mbox{$\forces_{P_\eta}\dot Q_\eta$ is $\xi$-directed-closed.}
We may factorise $P_\ka$ as $P_\xi*P^{[\xi,\ka)}$;
we shall show that given $G_\xi=G\cap P_\xi$,
it is dense in $P^{[\xi,\ka)}$ to be a master condition for an embedding
from $\sigma_{G_\xi}$ to another element of $B$.

Towards that end, consider an arbitrary $\dot r$ forced to be in 
$\dot P^{[\xi,\ka)}$.
We have chosen $\xi$ such that $|P_\xi|\leq\xi$ and a direct limit
is taken at $\ka$, 
%$\ka$ is Mahlo, $P_\ka$ is $\ka$-cc 
%(see Baumgartner~\cite[Corollary~2.4]{Bau:IF}), so
hence,
we may take $\eta<\ka$ large enough that 
\[
\forces_{P_\xi}\dot r\in\dot P^{[\xi,\eta)}
\]
(see for example Jech~\cite[Lemma~21.8]{Jech:ST}) and 
\[
\forces_{P_\eta}\dot P^{[\eta,\ka)}\text{ is $|G_\xi|^+$-directed-closed}. 
\]
Let $j:V_\eta\to V_\la$ in $V$ be an elementary embedding witnessing the
$\eta$-extendibility of $\al$ below $\ka$ for $\dot B$.
In particular, this entails that $j(\al)>\eta$.
The cardinal $\al$
%Since $\al$ is $\eta$-extendible below $\ka$ for $\dot B$, it 
is 
%in particular
inaccessible, and so for any $q\in P_\eta$, $\supp(q)\cap\al$ is bounded
by some $\be<\al$.
Therefore, by elementarity and the fact that $\crit(j)=\al$, 
$\supp(j(q))\cap j(\al)$ is bounded by 
$\be$, and $j(q)\restr\be=q\restr\be$.
Note that this implies that $\supp(j(q))\cap[\xi,\eta)=\emptyset,$
and we can extend $j(q)$ to $j(q)\land r$, where
\[
(j(q)\land r)(\zeta)=
\begin{cases}
r(\zeta)&\text{if }\xi\leq\zeta<\eta\\
j(q)(\zeta)&\text{otherwise.}
\end{cases}
\]

Since $G_\xi$ is a filter, $j``G_\xi$ is directed,
so by the choice of $\eta$
there is a single master condition $g\in (P^{[j(\al),\ka)})^{V[G_\xi]}$
such that $g\leq j(q)\restr[j(\al),\ka)$ for all $q\in G_\xi$.
Thus, the condition $g\land r\in P^{[\xi,\ka)}$ extends $r$ and
forces that $j\restr V_\xi:V_\xi\to V_{j(\xi)}$ lifts to  
an embedding $j':V_\xi[G_\xi]\to V_{j(\xi)}[G_{j(\xi)}]$
defined by $j'(\genval{\dot x}{G_\xi})=\genval{(j(\dot x))}{G_{j(\xi)}}$,
which is well-defined and elementary by the definability of forcing.
Note that $V_\xi[G_\xi]=V_\xi^{V[G_\xi]}$ 
since $\xi$ is $P_\ka$-reflecting,
and similarly for $j(\xi)$ by elementarity.

Now consider $\langle\sigma,q\rangle\in\dot B$, chosen above.
We have $j(\langle\sigma,q\rangle)\in\dot B$ by the choice of $j$, that is,
$\langle j(\sigma),j(q)\rangle\in\dot B$.  Since $q\in G_\xi$ by assumption,
$g\land r\forces_{P^{[\xi,j(\xi))}}j(q)\in \dot G_{j(\xi)}$.
By the definability of satisfaction for models and the elementarity of $j'$,
we have
\[
g\land r\forces_{P^{[\xi,j(\xi))}}
j'\restr\genval{\sigma}{G}:
\genval{\sigma}{G}
\to j'(\genval{\sigma}{G})=\genval{j(\sigma)}{G}\text{ is elementary.}
\]
Of course, $j'\restr\genval{\sigma}{G}$ cannot be the identity, 
as the domain of $\genval{\sigma}{G}$ 
is at least $\al=\crit(j)$.
Thus
$g\land r$ extends $r$ and
forces there to be a non-trivial elementary embedding between 
two distinct elements of
$\genval{\dot B}{G}=A$, as was required.
\end{proof}

With Theorem~\ref{REInd} at our disposal, many relative consistency results
become immediate by standard techniques.
As an example, we list a few principles familiar from G\"odel's constructible
universe $L$.
\begin{coroll}\label{VcCon}
If the existence of \Vopenka cardinals is consistent, then the existence of
\Vopenka cardinals is also consistent with each of the following.
\begin{enumerate}
\item GCH
\item $V=\HOD$
\item $\diamondsuit^+_{\ka^+}$ holds for every infinite cardinal $\ka$.
\item\label{morass} Morasses exist at every uncountable non-\Vopenka cardinal.
\end{enumerate}
\end{coroll}
\begin{proof}
There are known reverse Easton iterations of increasingly 
directed closed forcings
to obtain each of the listed properties,
such that the tail of the iteration from any \Vopenka cardinal $\ka$
is $\ka^+$-directed closed.  
For GCH, see Jensen~\cite{Jen:MCG},
for $V=\HOD$ see 
Brooke-Taylor~\cite{Me:LCDWO} or
Asper\'o and Friedman~\cite{AsF:LCLDWOU},
for $\diamondsuit^+_{\ka^+}$ 
see Cummings, Foreman and Magidor~\cite[Section~12]{CFM:SSSR},
and for morasses see Velleman~\cite{Vell:MDF} or \cite{Vell:SiM} or 
Brooke-Taylor and Friedman~\cite{Me:LCMor}.
\end{proof}

Part~\ref{morass} of Corollary~\ref{VcCon} raises the following question.

\begin{qn}
Is it consistent, relative to the existence of a Vo\-p\v{e}n\-ka cardinal,
to have a morass at a \Vopenka cardinal?
\end{qn}

Note that \Vopenka cardinals generally do not have 
the kind of downward reflection properties that one usually expects of strong
large cardinals, as they are not themselves weakly compact in general.
Thus, it should not be surprising that
Theorem~\ref{REInd} can be used to make properties
that hold at $\ka$ fail everywhere below $\ka$.

We also observe that in the proof of Theorem~\ref{REInd},
the assumption of increasing directed-closure for the 
partial orders $\dot Q_\al$, as opposed to some weaker form of closure, was 
necessary.  In particular, using $\al$-closed forcings, one can obtain
$\al$-Kurepa trees on inaccessible $\al$, and $\al$ will not be
ineffable in the extension --- see Cummings~\cite[Section~6]{Cum:IFE}.
But there must be many ineffable cardinals below any \Vopenka cardinal, so
a reverse Easton iteration of such forcings must destroy all
\Vopenka cardinals.

\section{Definable \VPr}\label{DVP}

We now extend the results of the previous sections
to the definable class form of \VPr.  
The 
forcing partial orders used in this section will correspondingly 
not always be sets. 
Rather, in Theorem~\ref{VPREInd} we shall consider class-length
reverse Easton iterations of increasingly directed-closed set forcings.
However, such class forcings are very well behaved; they are
tame in the sense of Friedman~\cite{SDF:FSCF}, and
in particular they preserve ZFC and have a definable forcing relation
as for set forcing.

The first issue to address is that of names.  Thanks to the definability of the
forcing relation, we can have ground model ``names'' for classes in the 
extension, in the following sense.
\begin{lemma}
Let $V[G]$ be a (set- or tame class-) generic extension of $V$,
and let $A$ be a definable class in $V[G]$.
Then there is a definable class $\dot A$ in $V$ such that
for every $x\in V[G]$, $x\in A$ if and only if there is 
a $\langle\dot x, p\rangle\in\dot A$ such that $\genval{(\dot x)}{G}=x$ and
$p\in G$.
\end{lemma}
\begin{proof}
Suppose $A$ is of the form
\[
A=\{x\in V[G]\st \varphi(x,z)\}
\]
for some parameter $z\in V[G]$.  Fix a name $\dot z$ for $z$.
Then
\[
\dot A=
\{\langle\dot x,p\rangle\st p\forces\varphi(\dot x,\dot z)\}
\]
is as required.
\end{proof}
We shall refer to such an $\dot A$ as a \emph{class name}.

We will of course need to use definable class forms of 
some of the properties of \Vopenka cardinals that we have used,
but fortunately these are mostly provided in 
Solovay, Reinhardt and Kanamori~\cite{SRK:SAIEE}.
Using these results we can moreover prove the definable class 
version of Lemma~\ref{VPoLs} without assuming $V=\HOD$.
We begin with the analogue of Definition~\ref{extblBelkaA}.
\begin{defn}
Let $A$ be a proper class.
A cardinal $\al<\eta$ is
\emph{$\eta$-extendible for $A$} 
if there is some $\zeta$ and an elementary
embedding
$j:\langle V_\eta,\in,A\cap V_\eta\rangle\to
\langle V_\zeta,\in,A\cap V_\zeta\rangle$
with critical point $\al$ and $j(\al)>\eta$.
A cardinal $\al$ is \emph{$A$-extendible}
if it is $\eta$-extendible for $A$ for all $\eta>\al.$
\end{defn}
\begin{lemma}\label{VPeq}
The following are equivalent.
\begin{enumerate}
\item\label{VPeqVP} \VPr
\item\label{VPeqAext} 
For every proper class $A$ there is an $A$-extendible cardinal.
\item\label{VPeqOrd} 
For every proper class $A$ of ordinal $\calL$-structures, there
exist $\calM$ and $\calN$ in $A$ of different cardinalities 
such that there
is an elementary embedding $j:\calM\to\calN$.
\end{enumerate}
\end{lemma}
\begin{proof}
(\ref{VPeqVP})$\Implies$(\ref{VPeqAext}) here is (1)$\Implies$(2) of
Theorem~6.9 of Solovay, Reinhardt and Kanamori~\cite{SRK:SAIEE}; 
briefly, structures are constructed for each $\al$ with reference to the
least failure of $\eta$-extendibility of $\al$ for $A$,
and then a \VPr embedding for this class of structures would yield a
contradiction if there were no $A$-extendible $\al$.

For (\ref{VPeqAext})$\Implies$(\ref{VPeqOrd}), let $A$ be as in 
(\ref{VPeqOrd}), and let $\al$ be $A$-extendible.  
If $\calM\in A$ is such that $\dom(\calM)\geq\al$, and $j$ witnesses that
$\al$ is \mbox{$(\rank(\calM)+1)$}-extendible for $A$, then 
$j\restr\dom(\calM):\calM\to j(\calM)$ is elementary, and 
$j(\calM)\in A$ has cardinality 
$|j(\calM)| \neq|\calM|\geq\al$.

Finally, for (\ref{VPeqOrd})$\Implies$(\ref{VPeqVP}),
let $A$ be a class of $\Lstd$-structures.  
Let $B$ be the class of \emph{all} ordinal $\Lstd$-structures $\bar\calM$
such that $\bar\calM$ is isomorphic to some $\calM$ in $A$.
By the Axiom of Choice, there will be elements of $B$ isomorphic to any
given element of $A$, but since we are not choosing representatives,
we do not need to appeal to definable global choice, that is, $V=\HOD$.
Applying (\ref{VPeqOrd}) to $B$, we get members of $B$ with 
an elementary embedding between them which have different cardinalities,
and thus correspond to different elements of $A$.  We thus get an
elementary embedding between distinct members of $A$, as required.
\end{proof}

We are now ready to translate our results to the definable class setting.
The natural L\'evy--Solovay theorem for definable \VPr 
is with ``set-sized'' taking the place of ``small''.
\begin{thm}\label{setpres}
Suppose \VPr holds in $V$, 
$\P$ is a partially ordered set in $V$, and $G$ is $\P$-generic over $V$.  
Then \VPr holds in $V[G]$.
\end{thm}
\begin{proof}
The proof is as for Theorem~\ref{smallpres}.
Lemma~\ref{nicenames} is equally valid for class names $\dot A$, 
%yielding a class name $\dot B$,
using no more choice than a well-order on the set of antichains of $\P$.
This much choice is also sufficient for 
the thinning out of $\dot B_r$ to a name $\dot C$,
thanks to the simple form of the names in $\dot B_r$;
moreover, $\dot C$ can be taken such that for different $\sigma_0$
and $\sigma_1$ appearing in $\dot C$, 
$|\ga_{\sigma_0}|\neq|\ga_{\sigma_1}|$, so that Lemma~\ref{VPeq} will apply.
The family $D$ of structures of the proof of Theorem~\ref{smallpres} 
now becomes a proper class, and the rest of the proof goes through unchanged.
\end{proof}

We thus come to the definable form of our main theorem.
\begin{thm}\label{VPREInd}
Assume \VPr.
Suppose %$P=\underrightarrow{\lim}(\langle P_\al\st\al\in\ORD\rangle)$
\[
P=\lim_{\longrightarrow}(\langle P_\al\st\al\in\ORD\rangle)
\]
is the reverse Easton iteration of 
$\langle\dot Q_\al\st\al\in\ORD\rangle$,
where
for every ordinal $\ga$, there is an $\eta_0$ such that for all
$\eta\geq\eta_0$,
\[
\forces_{P_\eta}\dot Q_\eta\text{ is $\ga$-directed closed.}
\]
Then in any $P$-generic extension, \VPr holds.
\end{thm}
\begin{proof}
Theorem~6.6 of Solovay, Reinhardt and Kanamori~\cite{SRK:SAIEE}
shows that extendible cardinals, and so in particular Mahlo cardinals,
are stationary in $\ORD$, and so
the analogue of Lemma~\ref{PClosRefl} goes through.
Converting Lemma~\ref{nicelocnames} is unproblematic.
Lemma~\ref{VPeq} above (or indeed Theorem~6.9 of \cite{SRK:SAIEE}) 
gives the appropriate analogue of
Proposition~\ref{VopCardExtA}, and the rest of the proof translates smoothly.
\end{proof}

%Corollary~\ref{VcCon} now translates directly.
\begin{coroll}\label{VPCon}
If \VPr is consistent, then Vo\-p\v{e}n\-ka's Principle
is also consistent with each of the following.
\begin{enumerate}
\item GCH
\item $V=\HOD$
\item $\diamondsuit^+_{\ka^+}$ holds for every infinite cardinal $\ka$.
\item\label{VPmorass} Morasses exist at every uncountable cardinal.
\hfill\qedsymbol
\end{enumerate}
\end{coroll}

%\note{Work in Friedman and Thompson result \cite{FrT:IMGD}
%as an example somewhere.}

\bibliographystyle{asl}
\bibliography{logic}

\end{document}